\documentclass[
final
]{dmtcs-episciences}

\usepackage[utf8]{inputenc}
\usepackage{subfigure}
\usepackage{amsmath}
\usepackage[round]{natbib}

\author{Winfried Hochst\"{a}ttler\affiliationmark{1}
  \and Felix Schr\"{o}der\affiliationmark{2}
  \and Raphael Steiner\affiliationmark{2}\thanks{funded by GRK-2434 Facets of Complexity}}
\title[On the Complexity of Digraph Colourings and Vertex Arboricity]{On the Complexity of Digraph Colourings and Vertex Arboricity}
\affiliation{
  FernUniversit\"{a}t in Hagen, Germany\\
  Technische Universit\"{a}t Berlin, Germany}
\received{2019-1-30}
\revised{2019-12-13}
\accepted{2019-12-15}

\def\n{\par\noindent}
\def\N{\par\par\noindent}
\newtheorem{theorem}{Theorem}
\newtheorem{proposition}{Proposition}
\newtheorem{lemma}{Lemma}
\newtheorem{corollary}{Corollary}

\newtheorem{conjecture}{Conjecture}
\newtheorem{observation}{Observation}

\newtheorem{definition}{Definition}
\newtheorem{problem}{Problem}
\begin{document}
\publicationdetails{22}{2019}{1}{4}{5140}
\maketitle
\begin{abstract}
It has been shown by Bokal et al.~that deciding 2-colourability of
  digraphs is an NP-complete problem. This result was later on
  extended by Feder et al.~to prove that deciding whether a digraph
  has a circular $p$-colouring is NP-complete for all rational $p>1$.
  In this paper, we consider the complexity of corresponding decision
  problems for related notions of fractional colourings for digraphs
  and graphs, including the \emph{star dichromatic number}, the
  \emph{fractional dichromatic number} and the \emph{circular vertex
    arboricity}. We prove the following results:
\begin{itemize}
\item Deciding if the star
  dichromatic number of a digraph is at most $p$ is NP-complete for every rational $p>1$.
\item Deciding if the fractional dichromatic number of a digraph is at most $p$ is NP-complete for every $p>1, p \neq 2$.
\item Deciding if the circular vertex arboricity of a graph is at most $p$ is NP-complete for every rational $p>1$.
\end{itemize}    
  To show these results, different techniques
  are required in each case. In order to prove the first
  result, we relate the star dichromatic number to a new notion of
  homomorphisms between digraphs, called \emph{circular
    homomorphisms}, which might be of independent interest. We provide
  a classification of the computational complexities of the
  corresponding homomorphism colouring problems similar to the one
  derived by Feder et al.~for acyclic homomorphisms.
\end{abstract}

Graphs and digraphs in this paper are considered loopless, but are allowed to have multiple parallel and anti-parallel edges/arcs between vertices. We will refer to edges $e$ in graphs by $uw$ where $u,w$ are the end vertices of $e$, if this does not lead to confusion with parallel edges. Given an arc $e$ of a digraph, we use $e=u \rightarrow w$ or equivalently $e=(u,w)$ to express that $e$ has tail $u$ and head $w$. This is not to be understood as a proper equality but as a statement on the arc $e$. Cycles and paths in graphs and directed cycles and paths in digraphs are always considered without repeated vertices. A cycle of length two in a graph consists of two parallel edges, while a directed cycle of length two (a \emph{digon}) is a pair of anti-parallel arcs in a digraph.

Throughout the paper, whenever we write $a+b$ or $a-b$ for elements $a,b \in \mathbb{Z}_k$, this is meant as on the group $(\mathbb{Z}_k,+)$. In most cases, we will identify the elements of $\mathbb{Z}_k$ with their representatives within $\{0,\ldots,k-1\}$. In addition, for every $k \in \mathbb{N}$ and
elements $x,y \in \{0,\ldots,k-1\} \simeq \mathbb{Z}_k$, let $\textrm{dist}_k(x,y):=|(x-y)
\textrm{mod } k|_k,$ where $|a|_k:=\min\{|a|,|k-a|\},$ for all
$a=0,\ldots,k-1$, denote the {\em circular $k$-distance between $x$ and~$y$}. For elements $x$ of $\mathbb{Z}_k$ or $\mathbb{Z}$ we will use $x \textrm{ mod }k$ to denote the unique element within $\{0,\ldots,k-1\} \subseteq \mathbb{Z}$ equivalent to $x$ modulo $k$. \\
Circular colourings of graphs were introduced by
\cite{vince}, where the concept of the \emph{star chromatic number},
nowadays also known as the \emph{circular chromatic number} of a
graph, made its first appearance.
The original definition of the star chromatic number by Vince is
based on so-called $(k,d)$-colourings, where colours at adjacent
vertices are not only required to be distinct as usual but moreover
'far apart' in the following sense:

\begin{definition}[\cite{vince}]
  Let $G$ be a graph and $(k,d) \in \mathbb{N}^2, k \ge d$. A
  {\em $(k,d)$-colouring of $G$} is an assignment $c:V(G)
  \rightarrow \{0,\ldots,k-1\} \simeq \mathbb{Z}_k$ of colours to the
  vertices so that $\textrm{dist}_k(c(u),c(w)) \ge d$ whenever $u,w$ are
  adjacent. \n
  The \emph{circular chromatic number} $\chi_c(G) \ge 1$ of the graph $G$ is defined as the infimum over all values of $\frac{k}{d}$ for which $(k,d)$-colourings exist. 
\end{definition}
The following canonical construction related to the circular chromatic number will be used in Section \ref{ComplVertA}.
\begin{definition} \label{circulant}
For any given natural numbers $(k,d) \in \mathbb{N}^2$ with $k \ge 2d$, we denote by $C(k,d)$ the circulant graph with vertex set $\mathbb{Z}_k$ where vertices $i \neq j \in \mathbb{Z}_k$ are adjacent if and only if $\textrm{dist}_k(i,j) \ge d$. 
\end{definition}
As was observed in Theorem 6 by \cite{vince}, it holds that $\chi_c(C(k,d))=\frac{k}{d}$ and $\chi(C(k,d))=\lceil \frac{k}{d} \rceil$.

In this paper, we focus on fractional colourings related to (directed) cycles in graphs and digraphs and determine the computational complexities of natural decision problems for various fractional colouring parameters. For details on circular colourings of graphs we refer to the survey article of \cite{xuding}.
\paragraph{Circular Chromatic Number of Digraphs}
Given a natural number $k \ge 1$, a \emph{$k$-colouring} of a digraph $D$ with $k$ colours is defined to be an assignment $c:V(D) \rightarrow \{0,\ldots,k-1\}$ with the property that there is no monochromatic directed cycle, i.e., $c^{-1}(i)$ induces an acyclic subdigraph of $D$ for every $i \in \{0,\ldots,k-1\}$. The minimal number of colours required to colour a digraph $D$ in this way is defined to be the \emph{dichromatic number}~$\vec{\chi}(D)$. This notion, introduced by \cite{ErdNeuLara} and \cite{NeuLara} generalises graph colourings and has been studied in numerous papers until today with still a lot of natural problems remaining unresolved.  \n
Circular colourings of digraphs were introduced by \cite{bokal} as a finer distinction between digraphs with the same dichromatic number by allowing fractional values. Instead of integers, they allow real numbers as colours in their definition: \n Given a real number $p \ge 1$, consider a plane-circle $S_p$ of perimeter $p$ and define a \emph{weak circular $p$-colouring} as a colour-map $c:V(D) \rightarrow S_p$, such that equal colours at both ends of an arc, i.e., $c(u)=c(w)$ where $e=(u,w) \in E(D)$, are allowed, but at the same time, the clockwise distance from $c(u)$ to $c(w)$ on $S_p$ is at least $1$ whenever they are distinct. Moreover, each so-called colour class, i.e., $c^{-1}(t), t \in S_p$, has to induce an acyclic subdigraph of~$D$. \n
The circular dichromatic number $\vec{\chi}_c(D)$ now is defined as the infimum over all values $p \ge 1$ providing weak circular $p$-colourings of~$D$. This infimum can be proved to be always attained as a minimum. \n
For any natural numbers $k \ge d \ge 1$, denote by $\vec{C}(k,d)$ the digraph with vertex set $\mathbb{Z}_k$ in which there is an arc $(i,j)$ for $i,j \in \mathbb{Z}_k$ if and only if $(j-i) \textrm{ mod }k \ge d$. The following sums up the most basic properties of the circular dichromatic number and these special circulant digraphs.
\newpage
\begin{theorem}[\cite{bokal}, \cite{mastert}] \label{PropertiesCircularDi}
Let $D$ be a digraph. Then the following holds:
\begin{enumerate}
\item[(i)] $\vec{\chi}_c(D) \ge 1$ is a rational number with numerator at most~$|V(D)|$.
\item[(ii)] $\lceil \vec{\chi}_c(D) \rceil=\vec{\chi}(D)$, i.e., $\vec{\chi}_c(D) \in (\vec{\chi}(D)-1,\vec{\chi}(D)]$.
\item[(iii)] $\vec{C}(k,d)$ admits circular dichromatic number exactly $\frac{k}{d}$ for any $k \ge d \in \mathbb{N}$.
\item[(iv)] Any digraph $D$ is weakly circularly $p$-colourable for $p \ge 1$ if and only if for every $(k,d) \in \mathbb{N}^2$ with $\frac{k}{d} \ge p$, $D$ admits a colouring $c_{k,d}:V(D) \rightarrow \mathbb{Z}_k$ with the following properties: For any arc $(u,w) \in E(D)$, either $c_{k,d}(u)=c_{k,d}(w)$ or $(c_{k,d}(w)-c_{k,d}(u)) \textrm{ mod }k \ge d$, and $c_{k,d}^{-1}(i)$ is acyclic for every $i \in \mathbb{Z}_k$.
\end{enumerate}
\end{theorem}
\paragraph{Graph Homomorphisms and Acyclic Homomorphisms}
Given a pair of graphs $G,H$, a graph homomorphism from $G$ to $H$ is a mapping $\phi:V(G) \rightarrow V(H)$ which preserves adjacency. It is well-known that graph homomorphisms generalise graph colourings in the following way: Given a fixed graph $H$, for any graph $G$, an $H$-colouring is defined to be a graph homomorphism $\phi:V(G) \rightarrow V(H)$. The \emph{$H$-colouring problem} then asks for a given graph $G$ whether it is $H$-colourable. If we take $H$ to be the complete graph on $k$ vertices, this is just the $k$-colouring problem for graphs which is known to be polynomially solvable for $k=2$ and NP-complete for $k \ge 3$. It was a long-standing open problem to determine the complexity of $H$-colourability for arbitrary graphs~$H$. This was finally resolved by \cite{Hcolouring} who proved the following:
\begin{theorem}[\cite{Hcolouring}] \label{nesetril}
The $H$-colouring problem is polynomially solvable if $H$ is bipartite, and it is NP-complete if $H$ is non-bipartite.
\end{theorem}
It is natural to ask for a definition of homomorphisms acting on digraphs which resembles digraph colourings in a similar way. One such notion which has received quite some attention in past years are \emph{acyclic homomorphisms}. Given a pair $D_1,D_2$ of digraphs, an acyclic homomorphism from $D_1$ to $D_2$ is defined to be a mapping $\phi:V(D_1) \rightarrow V(D_2)$ with the property that for any arc $(u,w)$ in $D$, either $\phi(u)=\phi(w)$ or $(\phi(u),\phi(w))$ is an arc in $D_2$, and additionally, for every vertex $v \in V(D_2)$, the vertex set $\phi^{-1}(v)$ induces an acyclic subdigraph of~$D_1$. The following statement describes the relation of (circular) digraph colourings and acyclic homomorphisms and shows that for digraph colourings, the circulant digraphs $\vec{C}(k,d)$ as defined above take the role of the complete graphs for usual graph colourings. 
\begin{proposition}[\cite{bokal}] \label{homcolbojan}
Let $p=\frac{k}{d} \ge 1$ be a rational number. Then for any digraph $D$, we have $\vec{\chi}_c(D) \leq p$ if and only if there is an acyclic homomorphism mapping $D$ to the digraph~$\vec{C}(k,d)$.
\end{proposition}
The question of determining the complexity of the decision problem whether or not $\vec{\chi}_c(D) \leq p$ for different real numbers $p \ge 1$ was raised by \cite{bokal} and answered in \cite{bojannp} for rational values of $p$ in form of a much more general statement which can be seen as a variant of Theorem~\ref{nesetril} for acyclic homomorphisms:
\begin{theorem}[\cite{bojannp}] \label{mainbojan}
Let $F$ be a digraph. Then the $F$-colouring problem, i.e., the problem of deciding whether a given digraph $D$ admits an acyclic homomorphism to $F$, is polynomially solvable if $F$ is acyclic and NP-complete otherwise. 
\end{theorem}
\begin{corollary}[\cite{bojannp}]\label{subbojan}
Given a rational number $p>1$, deciding whether a digraph admits a weak circular $p$-colouring, i.e., $\vec{\chi}_c(D) \leq p$, is NP-complete. 
\end{corollary}
\paragraph{Star Dichromatic Number}
Another related concept of circular colourings of digraphs was introduced by \cite{hochsteiner} under the name of the \emph{star dichromatic number} $\vec{\chi}^\ast(D)$ of a digraph. Again, for a colouring, real numbers associated with a plane circle are used, but instead of looking at circular distances between adjacent vertices, an \emph{acyclic $p$-colouring} of a digraph $D$ for any $p \ge 1$ requires pre-images of cyclic open subintervals of length $1$ to induce acyclic subdigraphs. Alternatively, one may use pairs of integers to define the star dichromatic number:
\begin{definition}[\cite{hochsteiner}] \label{DefDir} Let $D$ be a digraph, $(k,d) \in
  \mathbb{N}^2, k \ge d$. An \emph{acyclic $(k,d)$-colouring} of $D$
  is an assignment $c:V(D) \rightarrow \mathbb{Z}_k$ of colours to the
  vertices such that for every $i \in \mathbb{Z}_k$, the pre-image of
  the cyclic interval $A_i:=\{i,i+1,\ldots,i+d-1\} \subseteq
  \mathbb{Z}_k$ of colours, $c^{-1}(A_i) \subseteq V(D)$, induces an
  acyclic subdigraph of~$D$. The infimum over the values $\frac{k}{d}$ for which an acylic $(k,d)$-colouring exists is defined to be   the \emph{star dichromatic number} $\vec{\chi}^\ast(D) \ge 1$ of $D$. 
\end{definition}
Similar to the circular dichromatic number, the star dichromatic number fulfils a series of natural properties.
\begin{theorem}[\cite{hochsteiner}] \label{StarDichBasic}
Let $D$ be a digraph. Then the following holds:
\begin{enumerate}
\item[(i)] $\vec{\chi}^\ast(D) \ge 1$ is a rational number with numerator at most~$|V(D)|$.
\item[(ii)] $\lceil \vec{\chi}^\ast(D) \rceil=\vec{\chi}(D)$, i.e., $\vec{\chi}^\ast(D) \in (\vec{\chi}(D)-1,\vec{\chi}(D)]$.
\item[(iii)] $\vec{C}(k,d)$ admits star dichromatic number exactly $\frac{k}{d}$ for any $k \ge d \in \mathbb{N}$.
\item[(iv)] For all $k \ge d \in \mathbb{N}$, $D$ admits an acyclic $(k,d)$-colouring if and only if $\vec{\chi}^\ast(D) \leq \frac{k}{d}$.
\end{enumerate}
\end{theorem}
Although the star dichromatic number and the circular dichromatic number have a similar definition, they may behave very differently. While the star dichromatic number is immune to the addition of sinks and sources (as directed cycles may never pass them), this may have a significant effect on the circular dichromatic number, see Figure \ref{starcirc} for an illustration.

One of the by-products of the Section \ref{star} will be a notion of
homomorphisms for digraphs, so-called \emph{circular homomorphisms},
which are appropriate for generalising the star dichromatic number in
the same way acyclic homomorphisms generalise circular digraph
colourings.
\paragraph{Fractional Dichromatic Number}
The last notion of fractional colourings for digraphs we want to discuss here is the \emph{fractional dichromatic number} of a digraph $D$, denoted by $\vec{\chi}_f(D) \ge 1$. As its analogue for graphs, the well-known \emph{fractional chromatic number} $\chi_f(G)$ of a graph $G$, it may be defined as the optimal value of a linear program. Here, acyclic vertex sets play the role of independent vertex sets.
\begin{figure}[h] 
\centering
\includegraphics[scale=1]{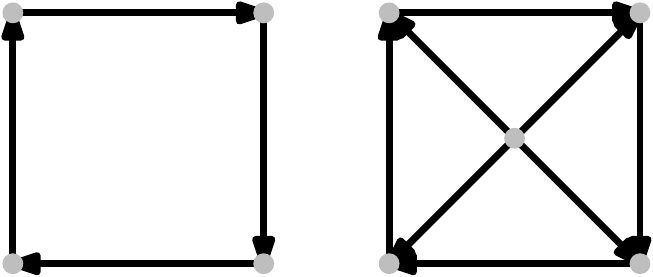}
\caption{Left: The directed cycle $\vec{C}_4$, which has fractional, star and circular dichromatic number $\frac{4}{3}$. While the addition of a dominating source does not change the fractional and the star dichromatic number, the circular dichromatic number jumps to $2$ (Right).}\label{starcirc}
\end{figure}
\begin{definition}[\cite{doct} and \cite{frdichr}] \noindent
  Let $D$ be a digraph. Denote by $\mathcal{A}(D)$ the collection of
  vertex subsets of $D$ inducing an acyclic subdigraph, and for each
  $v \in V(D)$, let $\mathcal{A}(D,v) \subseteq \mathcal{A}(D)$ be the
  subset containing only those sets including~$v$.  The
  \emph{fractional dichromatic number} $\vec{\chi}_f(D)$ of $D$ is now
  defined as the value of 
 \begin{align} \label{primal}
\min &\sum_{A \in \mathcal{A}(D)}{x_A}\\
\textrm{  subj.\ to }&
\sum_{A \in \mathcal{A}(D,v)}{x_A} \ge 1, \textrm{ for all } v \in V(D) \cr
&x \ge 0.\nonumber
\end{align}
\end{definition} 
The fractional dichromatic number has received some attention in recent years, it has e.g.\ shown useful for proving a fractional version of the so-called \emph{Erd\H{o}s-Neumann-Lara-conjecture} (\cite{frdichr}) or in the light of acyclic homomorphisms (\cite{doct}).\n
By dualising the linear program defining $\vec{\chi}_f(D)$, we get the following alternative definition:
\begin{proposition}[\cite{hochsteiner}]\label{propdual}
The fractional dichromatic number of a digraph $D$ can be computed as the optimal value of
\begin{align}\label{dual}
\max &\sum_{v \in V}{y_v} \\
\textrm{ subj. to } &
\sum_{v \in A}{y_v} \leq 1, & \textrm{ for all } A \in \mathcal{A}(D) \nonumber\\
&y \ge 0.\nonumber
\end{align}
\end{proposition}
 The following puts the three presented fractional digraph colouring parameters in relation and establishes direct relations to corresponding notions for graphs. We need the following notation: For a given graph $G$, we denote by $S(G)$ its symmetric orientation, i.e., the digraph obtained from $G$ by replacing each edge by an antiparallel pair of arcs. 
\begin{theorem}[\cite{hochsteiner}] \label{relations}
\noindent
\begin{itemize}
\item[(i)] Let $D$ be a digraph. Then $\vec{\chi}_f(D) \leq \vec{\chi}^\ast(D) \leq \vec{\chi}_c(D)$.
\item[(ii)] For any graph $G$, we have $\vec{\chi}^\ast(S(G))=\vec{\chi}_c(S(G))=\chi_c(G)$ and $\vec{\chi}_f(S(G))=\chi_f(G)$.
\end{itemize}
\end{theorem}
\paragraph{Circular Vertex Arboricity}
The counterpart of digraph colourings for undirected graphs is known as the \emph{vertex arboricity}. Given some $k \in \mathbb{N}$, a \emph{$k$-tree-colouring} of a (multi-)graph $G$ is defined to be a colouring of the vertices of $G$ using colours $\{0,\ldots,k-1\}$ such that there are no monochromatic cycles, i.e., $G[c^{-1}(i)]$ is a forest for any $i \in \{0,...,k-1\}$. The \emph{vertex arboricity} $\textrm{va}(G)$, which was introduced by \cite{chartrand1}, then denotes the minimal number of colours required for a tree-colouring of the graph $G$. Since its introduction, this parameter, which is closely related to the Hamiltonicity of planar graphs, has been widely studied. Similar to the notions of circular colourings of graphs and digraphs, it is also possible to investigate a circular version of the vertex arboricity, which was introduced by \cite{fracvert} under the name \emph{circular vertex arboricity}. For this purpose, the notion of a \emph{$(k,d)$-tree-colouring} of a graph $G$ is defined. Similarly to acyclic $(k,d)$-colourings of digraphs, this is a mapping $c:V(G) \rightarrow \mathbb{Z}_k$ with the property that for any cyclic subinterval $A_i:=\{i,i+1,\ldots,i+d-1\}$ of $\mathbb{Z}_k$, the subgraph of $G$ induced by $c^{-1}(A_i)$ is a forest. The circular vertex arboricity $\textrm{va}_c(G)$ of the graph $G$ is now defined as the infimum of the values $\frac{k}{d}$ for which a $(k,d)$-tree-colouring exists. The circular vertex arboricity has the following basic properties.
\begin{theorem}[\cite{fracvert}]
Let $G$ be a (multi-)graph. Then the following holds:
\begin{enumerate}
\item[(i)] $\textrm{va}_c(G) \ge 1$ is a rational number with numerator at most~$|V(G)|$.
\item[(ii)] $\lceil \textrm{va}_c(G) \rceil=\textrm{va}(G)$, i.e., $\textrm{va}_c(G) \in (\textrm{va}(G)-1,\textrm{va}(G)]$.
\item[(iii)] For all $k \ge d \in \mathbb{N}$, $D$ admits a $(k,d)$-tree-colouring if and only if $\textrm{va}_c(G) \leq \frac{k}{d}$.
\end{enumerate}
\end{theorem}
\section{Outline}
The paper is divided into three sections, studying the complexity of decision problems for the \emph{Star Dichromatic Number}, the \emph{Fractional Dichromatic Number} respectively the \emph{Circular Vertex Arboricity} as defined above. The following main results are proved:

\begin{theorem} \label{cometogether}
\noindent
\begin{itemize}
\item For any fixed rational number $p>1$, deciding whether a given (multi-)digraph $D$ fulfills $\vec{\chi}^\ast(D) \leq p$ is NP-complete. 
\item For any fixed real number $p>1, p \neq 2$, deciding whether a given (multi-)digraph $D$ fulfills $\vec{\chi}_f(D) \leq p$ is NP-complete.
\item For any fixed rational number $p>1$, deciding whether a given (multi-)graph $G$ fulfills $\textrm{va}_c(G) \leq p$ is NP-complete.
\end{itemize}
\end{theorem}

This theorem answers open questions of \cite{hochsteiner} as well as questions in the context of the work by \cite{fracvert} and naturally extends the results achieved by \cite{bojannp} on
the circular dichromatic number. \n The proof of the NP-hardness in
each case requires different techniques. The notion of \emph{circular
  homomorphisms} acting between digraphs introduced in Section
\ref{star} might be of independent interest.
\section{The Complexity of the Star Dichromatic Number and Circular Homomorphisms} \label{star}
In this section, we deal with decision problems for the star dichromatic number analogous to those considered by \cite{bojannp}. The problem of determining the complexity of the following decision problem was posed by \cite{hochsteiner}:
\begin{problem} \label{probstar}
Let $p \ge 1$ be a fixed rational number. \n
Instance: A (multi-)digraph~$D$. \n Decide whether $\vec{\chi}^\ast(D) \leq p$.
\end{problem}
For $p=1$, the problem is to decide whether $\vec{\chi}^\ast(D)=1$,
which is equivalent to $D$ being acyclic, and this can be solved in
time linear in~$|E(D)|$. 
\N We now introduce \emph{circular homomorphisms} as an extension of
the well-known \emph{acyclic homomorphisms} defined in the previous
section. In the following, a vertex subset in a digraph will be called
\emph{acyclic} if the corresponding induced subdigraph is acyclic.
\begin{definition}\label{circhom}
  Let $D_1,D_2$ be digraphs. A mapping $\phi:V(D_1) \rightarrow
  V(D_2)$ is called a \emph{circular homomorphism}, if for all
  $A \subseteq V(D_2)$ such that $D_2[A]$ is acyclic, $\phi^{-1}(A)$ is acyclic in $D_1$.
  Equivalently, for any directed cycle $C$ in $D_1$, $D_2[\phi(V(C))]$
  contains a directed cycle.
\end{definition}
It is obvious that $D_1$ admits the injection $\textrm{id}|_{V(D_1)}$ as a circular homomorphism to $D_2$ whenever $D_1$ is a subdigraph of $D_2$ and that the composition of two circular homomorphisms remains a circular homomorphism. \n
Note that this definition is one natural way to generalise graph homomorphisms to digraphs, as the former may be characterised by the property that pre-images of \emph{independent} sets remain independent. This similarity is made precise by the following. \begin{proposition} \label{symmetricor}
Let $G_1, G_2$ be graphs and denote by $S(G_1)$, $S(G_2)$ its symmetric orientations. Then a mapping $\phi:V(G_1) \rightarrow V(G_2)$ is a graph homomorphism if and only if it is a circular homomorphism from $S(G_1)$ to~$S(G_2)$.
\end{proposition}
\begin{proof}
  This follows immediately from the characterisations of graph resp.\
  circular homomorphisms in terms of independent resp.\ acyclic vertex
  sets and the fact that for any graph $G$, the acyclic vertex sets of
  $S(G)$ are exactly the independent vertex sets of~$G$.
\end{proof}

The following, which is similar to Proposition \ref{homcolbojan},
reformulates acyclic $(k,d)$-colourings in terms of circular
homomorphisms.
\begin{proposition} \label{ckd}
Let $p=\frac{k}{d} \ge 1, k,d \in \mathbb{N}$. Then $\vec{\chi}^\ast(D) \leq p$, i.e., there is an acyclic $(k,d)$-colouring of $D$, if and only if there is a circular homomorphism from $D$ to~$\vec{C}(k,d)$.
\end{proposition}
\begin{proof}
Recall that $\vec{C}(k,d)$ was defined to be the digraph with vertex set $\mathbb{Z}_k \simeq \{0,\ldots,k-1\}$ where there is an arc $(i,j)$ between two elements if and only if $j-i \in \{d,\ldots,k-1\} \subseteq \mathbb{Z}_k$. To prove the claim, we need the following property: A vertex set $A \subseteq \vec{C}(k,d)$ is acyclic if and only if it is contained in a set of $d$ consecutive vertices, i.e., $A \subseteq \{i,i+1,\ldots,i+d-1\} \subseteq \mathbb{Z}_k$ with some $i \in \mathbb{Z}_k$. \n
For the first implication, assume that $A$ is acyclic. Then $\vec{C}(k,d)[A]$ must contain a sink $i \in A$ (i.e., $i$ has no out-neighbors in $A$), which means that none of the vertices $i+d,\ldots,i+k-1 \in \mathbb{Z}_k$ can be contained in $A$, and so $A \subseteq \{i,\ldots,i+d-1\}$. For the reverse, since $\vec{C}(k,d)$ is circulant, it is enough to show that $\{0,\ldots,d-1\} \subseteq \mathbb{Z}_k$ is acyclic. However, by definition, this interval can only contain backward arcs, and so the subdigraph of $\vec{C}(k,d)$ induced by $\{0,\ldots,d-1\}$ admits an acyclic ordering. \n
Consequently, the circular homomorphisms $\phi:V(D) \rightarrow \mathbb{Z}_k=V(\vec{C}(k,d))$ from any digraph $D$ to $\vec{C}(k,d)$ are exactly those mappings for which $\phi^{-1}(\{i,\ldots,i+d-1\})$ is acyclic for all $i \in \mathbb{Z}_k$, and this is just the same as an acyclic $(k,d)$-colouring of~$D$. This proves the claim.
\end{proof}

Furthermore, the well-studied acyclic homomorphisms between digraphs appear as a special case of circular homomorphisms:
\begin{proposition} \label{acyclich}
Let $D_1,D_2$ be two digraphs. Then any acyclic homomorphism $\phi:V(D_1) \rightarrow V(D_2)$ is a circular homomorphism.
\end{proposition}
\begin{proof}
Let $C$ be any directed cycle in~$D_1$. We need to show that $\phi(V(C))$ contains the vertex set of a directed cycle. Since pre-images of single vertices under $\phi$ are acyclic in $D_1$, $\phi(V(C))$ needs to contain at least two vertices. Any arc $x \rightarrow y$ on $C$ is either mapped to a single vertex $\phi(x)=\phi(y)$ or to an arc $\phi(x) \rightarrow \phi(y)$ of $D_2$, which implies that $D_2[\phi(V(C))]$ admits a closed directed trail visiting at least two vertices and thus also a directed cycle. This proves the claim.
\end{proof}

However, the reverse of this statement is not true. This follows from
the fact that there are digraphs with
$\vec{\chi}_c(D)>\vec{\chi}^\ast(D)$, Proposition \ref{homcolbojan}
and Proposition \ref{ckd}. Examples of such digraphs are e.g.\ directed
cycles with an additional dominating source (cf. \cite{hochsteiner}), see Figure \ref{starcirc}.

We conclude the discussion of circular homomorphisms with the
following observation, which identifies them as interlacing structures
between digraphs in terms of their star and fractional dichromatic
numbers:
\begin{proposition} \label{homandcol}
  Let $D_1,D_2$ be digraphs such that there is a circular homomorphism
  $\phi:V(D_1) \rightarrow V(D_2)$. Then $\vec{\chi}^\ast(D_1) \leq
  \vec{\chi}^\ast(D_2)$ and $\vec{\chi}_f(D_1) \leq
  \vec{\chi}_f(D_2)$.
\end{proposition}
\begin{proof}
The inequality for the star dichromatic number follows from Proposition \ref{ckd} and the fact that the composition of two circular homomorphisms remains a circular homomorphism. The inequality for the fractional dichromatic number can be seen from the definition in terms of the linear program (\ref{primal}) as follows: Given any optimal solution $x' \ge 0$ of the program with respect to $D_2$, define a corresponding instance $x \ge 0$ of the program for $D_1$ by assigning the value 
\begin{displaymath}x_A:=\sum_{\substack{A' \in \mathcal{A}(D_2):\\ \phi^{-1}(A')=A}}{x'_{A'}}\end{displaymath} for every acyclic vertex set $A \in \mathcal{A}(D_1)$. It is now easily verified using the fact that $\phi^{-1}(A') \in \mathcal{A}(D_1)$ for any $A' \in \mathcal{A}(D_2)$, that $x$ is a legal instance of (\ref{primal}) with respect to $D_1$ with 
\begin{displaymath}\vec{\chi}_f(D_1) \leq \sum_{A \in \mathcal{A}(D_1)}{x_A}=\sum_{A' \in \mathcal{A}(D_2)}{x'_{A'}}=\vec{\chi}_f(D_2).\end{displaymath} 
\end{proof}

Given a fixed digraph $F$, any other digraph $D$ will be called
\emph{circularly $F$-colourable} if there exists a circular
homomorphism mapping $D$ to~$F$. The following decision problem, which
can be seen as a directed analogue of the $H$-colouring problem for
graphs then generalises Problem \ref{probstar}.
\begin{problem}
Let $F$ be a fixed (multi-)digraph. \n
Instance: A (multi-)digraph $D$. \n
Decide whether $D$ is circularly $F$-colourable. 
\end{problem}
As in the graph colouring problem, there is a trivial case: Only acyclic digraphs map circularly to acyclic digraphs:
\begin{observation}
The circular $F$-colouring problem is polynomially solvable for any acyclic digraph $F$.
\end{observation}
We conjecture that this simple observation covers already all polynomially solvable cases under the assumption P$\neq$NP. In other words,
\begin{conjecture}\label{conj}
Let $F$ be a digraph which contains a directed cycle. Then the circular $F$-colouring problem is NP-complete. 
\end{conjecture}
Our main result of this section is the following theorem, which shows that this conjecture holds true in almost all the cases. Given a (multi-)digraph, the \emph{symmetric part} of $D$ is defined to be the graph on the same vertex set as $D$ which contains an edge $xy$ if and only if there is an arc from $x$ to $y$ and from $y$ to $x$ in $D$. 
\begin{theorem}\label{main}
Let $F$ be a digraph containing a directed cycle such that at least one of the following holds:
\begin{enumerate}
\item[(i)] The symmetric part of $F$ is empty, i.e., $F$ is digon-free, or
\item[(ii)] The symmetric part of $F$ contains an odd cycle, or
\item[(iii)] $F$ is 2-colourable.
\end{enumerate}
  Then the circular $F$-colouring problem is NP-complete. 
\end{theorem}
\begin{proof}
  We start by observing that the problem is certainly contained in NP:
  Given any digraph $D$, a circular homomorphism from $D$ to $F$ can
  be used as an NP-certificate. Note that the digraph $F$ itself
  defines the colouring problem and is not considered as an instance,
  and so for checking whether a given mapping $\phi:V(D) \rightarrow
  V(F)$ defines a circular homomorphism, it suffices to compute at
  most constantly many inverse images under $\phi$ and
  verify that the corresponding subdigraphs of $D$ are indeed acyclic,
  which can be done in polynomial time in~$|V(D)|$. 
  \begin{enumerate}
  \item[(i), (ii)] Assume that the symmetric part of $F$ is either empty or
    non-bipartite. Let $k$ denote the digirth, i.e., the length of a
    shortest directed cycle, of $F$. We will define a graph $H_F$
    whith vertex set $V(F)$ and in which two vertices $u \neq v$
    are adjacent if and only if there is a directed cycle of length $k$
    containing both $u$ and $v$ in $F$. Similar to the proof of the
    corresponding result by \cite{bojannp} (cf. Theorem
    \ref{mainbojan}), we will give a polynomial reduction of the
    $H_F$-colouring problem to the circular $F$-colouring problem. For a given
    instance $G$ of the $H_F$-colouring problem, we will construct a
    polynomial-sized instance $D_G$ for the circular $F$-colouring problem
    and prove that $G$ is $H_F$-colourable if and only if $D_G$ is
    circularly $F$-colourable. \n If the symmetric part of $F$ is
    empty, then there is a directed cycle of length $k \ge 3$ in $F$,
    which forms a clique of size $k$ in $H_F$ and therefore $H_F$
    contains a triangle. Otherwise, the symmetric part of $H_F$ is
    non-empty but non-bipartite. In this case, we have $k=2$ and the
    vertices of any odd cycle in the symmetric part of $F$ will form
    an odd cycle in $H_F$. In any case, $H_F$ is non-bipartite, and
    thus, the decision problem of $H_F$-colourability is NP-hard
    according to Theorem \ref{nesetril}.  \N Let now $G$ be an
    instance of the $H_F$-colouring problem. As $H_F$ is loopless, we
    may assume that the same holds for $G$ as well. We construct the
    digraph $D_G$ by first choosing some acyclic orientation $\vec{G}$
    of $G$ and then attaching to every arc $x \rightarrow y$
    of $\vec{G}$ a directed path of length $k-1$ in reverse direction
    whose only common vertices with $\vec{G}$ are $x$ and $y$, so that
    each arc $e=x \rightarrow y$ in $\vec{G}$ is contained in a
    directed cycle $C(e)$ of length~$k$ in $D_G$. The set of the $k-2$
    extra vertices added is pairwise disjoint for distinct edges.
    Clearly, this construction is polynomial in the size of $G$. \n
    We now claim that there is a graph homomorphism from $G$ to $H_F$
    if and only if there is a circular homomorphism from $D_G$ to $F$,
    which yields the desired polynomial reduction.  \n For the first
    implication let $\phi:V(G) \rightarrow V(H_F)$ be a graph
    homomorphism mapping $G$ to $H_F$.  Then for any arc
    $e=x \rightarrow y$ in $\vec{G}$, $\phi(x)\phi(y)$ is an edge in
    $H_F$, i.e., there is a directed cycle $C'(e)$ of length
    $k$   containing $\{\phi(x),\phi(y)\}$ in $F$. Moreover, since
    $|V(C(e)) \backslash \{x,y\}|=|V(C'(e)) \backslash
    \{\phi(x),\phi(y)\}|=k-2$, we find that there are bijections $f_e:V(C(e))
    \backslash \{x,y\} \rightarrow V(C'(e)) \backslash
    \{\phi(x),\phi(y)\}$ for every edge $e$ of~$G$.\n Let now
    $\phi':V(D_G) \rightarrow V(F)$ be the mapping defined by
    $\phi'(u):=\phi(u)$ for any $u \in V(G) \subseteq V(D_G)$ and
    $\phi'(u):=f_e(u)$ for any $u \in V(C(e)) \backslash \{x,y\}$ and
    any edge $e=xy \in E(G)$. We claim that this defines a circular
    homomorphism from $D_G$ to $F$: If $C$ is any directed cycle in
    $D_G$, since $\vec{G}$ is an acyclic orientation, $C$ needs to
    contain a full attachment path and thus the vertex set of $C(e)$
    for at least one edge $e \in E(G)$. This implies $\phi'(V(C))
    \supseteq \phi'(V(C(e))=V(C'(e))$ by definition of $\phi'$, and
    thus $C'(e)$ is a directed cycle contained in the image of $V(C)$,
    as required.  \n Conversely, assume there is a circular
    homomorphism $\phi'$ mapping $D_G$ to~$F$. We claim that the
    restriction $\phi:=\phi'|_{V(G)}$ is a graph homomorphism from $G$
    to~$H_F$. For this purpose, let $e=xy$ be any edge of~$G$. Then,
    since $\phi'$ is a circular homomorphism, $\phi'(V(C(e)))$
    contains the vertex set of a directed cycle in $F$, which must
    have length at least~$k$.  However, $|\phi'(V(C(e)))| \leq
    |V(C(e))|=k$ by definition of $C(e)$, so this directed cycle has
    exactly $\phi'(V(C(e)))$ as vertex set, which contains $\phi'(x)$
    and~$\phi'(y)$. Moreover, $\phi'$ restricted to $V(C(e))$ must be
    an injection, and so $\phi'(x) \neq \phi'(y)$. According to the
    definition of $H_F$, this finally implies that
    $\phi(x)\phi(y)=\phi'(x)\phi'(y)$ is an edge of $H_F$, and so
    $\phi$ is indeed a graph homomorphism as required.  This settles
    the poof in the case where the symmetric part of $F$ is empty or
    non-bipartite.
  \item[(iii)] Now, let $F$ be 2-colourable. Referring to (i) and (ii), we may
    assume that the symmetric part of $F$ is non-empty and bipartite.
    Since $F$ is 2-colourable, by Proposition \ref{ckd} there is a
    circular homomorphism from $F$ to~$\vec{C}_2$. On the other hand,
    since the symmetric part of $F$ is non-empty, $F$ contains a
    digon, and thus, there also is a circular homomorphism from
    $\vec{C}_2$ to~$F$. Hence, in this case the circular $F$-colouring
    and the circular $\vec{C}_2$-colouring problem are equivalent.
    However, deciding 2-colourability of a digraph is NP-hard (cf.~\cite{bokal}).
  \end{enumerate}
\end{proof}

Applying this result to the star dichromatic number, we finally obtain the desired hardness result:
\begin{theorem}
  Let $p>1$ be a rational number. Deciding whether $\vec{\chi}^\ast(D)
  \leq p$ is NP-complete.
\end{theorem}
\begin{proof}
  Let $p=\frac{k}{d}$ with $k,d \in \mathbb{N}$. By Proposition
  \ref{ckd}, the decision problem is equivalent to the circular
  $\vec{C}(k,d)$-colouring problem. \n To prove NP-completeness, we
  distinguish between $p \leq 2$ and $p>2$. In the first case,
  $\vec{C}(k,d)$ is not acyclic and 2-colourable and thus the claim
  follows from Theorem \ref{main}. \n In the case $p>2$, the symmetric
  part of $\vec{C}(k,d)$ is given by the circulant graph $C(k,d)$
  defined in Definition \ref{circulant}. $C(k,d)$ has chromatic number
  $\lceil p \rceil \ge 3$ and thus is not bipartite. Consequently,
  another application of Theorem \ref{main} yields the claimed result.
\end{proof}
\section{The Complexity of the Fractional Dichromatic Number} \label{secfrac}
As in the previous section, we now want to deal with decision problems for the fractional dichromatic number as follows.
\begin{problem} \label{probfrac}
Let $p \ge 1$ be a fixed real number. \n
Instance: A (multi-)digraph $D$. \n Decide whether $\vec{\chi}_f(D) \leq p$.
\end{problem}
Again, it is clear that for $p=1$, $\vec{\chi}_f(D) \leq p$ if and only if the digraph $D$ is acyclic, and this can be decided in polynomial time in $|V(D)|$. Conversely, we want to show in the following that for all real numbers $p>1, p \neq 2$, this problem is NP-complete. It is indeed always contained in NP:
\begin{observation}
For any $p \ge 1$, the Problem \ref{probfrac} is in NP.
\end{observation}
\begin{proof}
  Let $D$ be a digraph given as an instance of the problem. Let again
  $\mathcal{A}(D)$ denote the set of acyclic vertex sets in~$D$. We
  have to prove the existence of a certificate polynomially-sized in
  $|V(D)|$ which is verifiable in polynomial time. For this purpose,
  we repeat some standard arguments from linear programming for
  (\ref{primal}). Clearly, any optimal solution of (\ref{primal})
  satisfies $x \leq \textbf{1}$. Thus, adding the constraints $x_A \leq
  2$ for all $A \in \mathcal{A}(D)$ yields an equivalent bounded
  feasible program. As the optimal solution is attained by a vertex
  $x$ of the corresponding polyhedron there is a subset
  $|\mathcal{A}(D)|$ of the inequality-constraints which are satisfied
  by $x$ with equality, and the corresponding linear system uniquely
  determines~$x$. Since any optimal solution $x$ satisfies $x \leq 1$,
  none of the additional constraints is tight, implying that the
  support of $x$ satisfies $m:=|\textrm{supp}(x)| \le |V(D)|$. Denote by
  $x' \in \mathbb{R}^m$ the subvector of $x$ restricted to the
  support.  As $x'$ is the unique solution of a regular linear system,
  according to Cramer's rule, there are matrices $B_1,\ldots,B_m,B \in
  \{-1,0,1\}^{m \times m}$ such that
  $x_i=\frac{\textrm{det}(B_i)}{\textrm{det}(B)}, i=1,\ldots,m$. According
  to Hadamard's inequality, we have $|\textrm{det}(B_i)|,|\textrm{det}(B)|
  \leq m^{m/2} \leq |V(D)|^{|V(D)|/2}, i=1,\ldots,m$. This finally
  implies that there exist optimal solutions to the linear program
  (\ref{primal}) whose support is of size at most $|V(D)|$ and where
  the non-zero values in the solution are rational numbers, each of
  which can be stored using at most $|V(D)|\log |V(D)|$ bits. Such a
  solution can thus be described using $\mathcal{O}(|V(D)|^2\log
  |V(D)|)$ bits. As we can verify all the constraints and the
  inequality $\sum_{i=1}^{m}{x_i} \leq p$, certifying that
  $\vec{\chi}_f(D) \leq p$, in polynomial time in $m \leq |V(D)|$,
  this finally proves that we can use optimal solutions of this form
  as NP-certificates. This concludes the proof.
\end{proof}

We start our proof of the hardness with the following simple observation derived from the relation of the fractional chromatic and the fractional dichromatic number:
\begin{observation} \label{Gr2Fall}
Let $p \in \mathbb{R}, p > 2$. Then Problem \ref{probfrac} is NP-complete. 
\end{observation}
\begin{proof}
  It is well-known (see e.g.\ \cite{kneser}, Theorem 3.9.2) that the
  problem of deciding whether $\chi_f(G) \leq p$ for a given graph $G$
  is NP-hard for any real number~$p>2$. However, this problem
  admits a polynomial reduction to Problem \ref{probfrac} for $p$: For
  any graph $G$, the symmetric orientation $S(G)$ fulfils $\chi_f(G)
  \leq p \Leftrightarrow \vec{\chi}_f(S(G))=\chi_f(G) \leq p$. This
  proves the claim.
\end{proof}

It thus suffices to prove the hardness in the case $p \in (1,2)$. For any given $p$, we will reduce one of the decision problems proved to be hard in Observation \ref{Gr2Fall} to Problem \ref{probfrac} with~$p$. For this purpose, we introduce a certain operation on digraphs reducing its fractional dichromatic number:
\begin{definition}
Let $D$ be a digraph. For every $l \ge 1$, we denote by $D_l$ a digraph called \emph{$l$-split} of $D$ obtained from $D$ by replacing each vertex by a directed path of length $l-1$ as follows: Each vertex $x \in V(D)$ is assigned a directed path $P(x)=x_1 \rightarrow \ldots \rightarrow x_l$ in~$D_l$. The remaining adjacencies within $D_l$ are given as follows: For each arc $e=(u,w)$ in $D$, we have a corresponding arc $(u_l,w_1)$ in~$D_l$. Thus, in a path $P(x)$, $x_1,\ldots,x_{l-1}$ have outdegree $1$ while $x_2,\ldots,x_{l}$ have each exactly one incoming arc. It is furthermore obvious that each directed cycle in $D_l$ contains the whole path $P(x)$ or none of its vertices, for all $x \in V(D)$. This means that there is a bijection between the directed cycles in $D$ and those in $D_l$ by replacing each vertex $x \in V(D)$ contained in a directed cycle by $P(x)$ in $D_l$ and vice versa. 
\end{definition}
The following makes the relation between the fractional and star dichromatic numbers of $D$ and $D_l$ precise.
\begin{proposition}
For each digraph $D$, the following holds:
\begin{itemize}
\item[(i)] \begin{displaymath}\vec{\chi}^\ast(D_l) \leq \frac{l\vec{\chi}^\ast(D)}{(l-1)\vec{\chi}^\ast(D)+1}.\end{displaymath} 
\item[(ii)] \begin{displaymath} \vec{\chi}_f(D_l)=\frac{l\vec{\chi}_f(D)}{(l-1)\vec{\chi}_f(D)+1}.\end{displaymath}
If $\vec{\chi}^\ast(D)=\vec{\chi}_f(D)$, then $\vec{\chi}^\ast(D_l)=\vec{\chi}_f(D_l)=\frac{l\vec{\chi}^\ast(D)}{(l-1)\vec{\chi}^\ast(D)+1}$.
\end{itemize}
\end{proposition}
\begin{proof} \noindent
\begin{itemize}
\item[(i)]  In the following, let $p:=\vec{\chi}^\ast(D),
  q:=\vec{\chi}^\ast(D_l)$. We have to prove that \[q \leq
  \frac{lp}{(l-1)p+1}=\frac{p}{p-\frac{p-1}{l}}.\] By scaling, it
  suffices to construct a mapping $c_l:V(D_l) \rightarrow [0,p)$ such
  that there is no directed cycle in $D_l$ whose image is contained in
  some open subinterval of length $p-\frac{p-1}{l}$ of $D_l$. \n For
  this purpose, consider an acyclic $p$-colouring $c:V(D) \rightarrow
  [0,p)$ of $D$. Define $c_l$ according to
  $c_l(x_i):=(c(x)+(i-1)\frac{p-1}{l}) \textrm{ mod }p$ for all $x_i \in
  V(P(x))$ and $x \in V(D)$. \n In order to see that this defines a
  colouring as required, assume for contrary there was an open
  subinterval $(a,b)_p \subseteq [0,p)$ of length $p-\frac{p-1}{l}$
  containing the image $c_l(V(C_l))$ of a directed cycle $C_l$.
  Without loss of generality, let $a=0, b=p-\frac{p-1}{l}$. We claim
  that the directed cycle $C$ in $D$ induced by $C_l$ is contained in
  an open cyclic subinterval of $[0,p)$ of length $1$, which finally
  contradicts the definition of~$c$. Let $\overline{x} \in V(C)$ so
  that $c(\overline{x})=\max_{x \in V(C)} c(x)$. It now suffices to
  show that $c(\overline{x})<1$. We have that
  \[c_l(\overline{x}_i)=\left(c(\overline{x})+(i-1)\frac{p-1}{l}\right) \textrm{ mod
  }p \in \left(0,p-\frac{p-1}{l}\right) (i=1,\ldots,l).\] Because of
  $(c_l(\overline{x}_{i+1})-c_l(\overline{x}_i)) \textrm{ mod
  }p=\frac{p-1}{l}$ for $i=1,\ldots,l-1$, the $c_l(\overline{x}_i)$
  can not cross the gap $[p-\frac{p-1}{l},0]_p$ of length
  $\frac{p-1}{l}$ and thus we have
 \begin{displaymath}c_l(\overline{x}_{l})=c(\overline{x})+(l-1)\frac{p-1}{l}<p-\frac{p-1}{l}\end{displaymath} implying $c(\overline{x})<1$. Finally, this proves the first
  inequality. \n
\item[(ii)] For the second statement we again use the alternative representation of $\vec{\chi}_f(D)$ as the maximal value of the dual program (\ref{dual}) in Proposition~\ref{propdual}. \n
Throughout the rest of the proof, the following relation between acyclic sets of $D$ and its $l$-split will be crucial: Define a mapping $f:\mathcal{A}(D_l) \rightarrow \mathcal{A}(D)$ such that for all $B \subseteq V(D_l)$, $f(B):=\{x \in V(D)|V(P(x)) \subseteq B\}$. Furthermore, define $g:\mathcal{A}(D) \rightarrow \mathcal{A}(D_l)$ by $g(A):=\bigcup_{x \in A}{V(P(x))}$. These mappings are well-defined due to the bijection between directed cycles in $D$ resp.\ $D_l$ described above. We clearly have $f \circ g=\textrm{id}_{\mathcal{A}(D)}$ and thus, $g$ is injective while $f$ is surjective. \n
We start by showing that $\vec{\chi}_f(D_l) \geq \frac{l\vec{\chi}_f(D)}{(l-1)\vec{\chi}_f(D)+1}$. For this purpose, let $y_v, v \in V(D)$ be an optimal instance for the dual program (\ref{dual}) for $D$, i.e., $\sum_{v \in V(D)}{y_v}=\vec{\chi}_f(D)$. We define an instance of the dual problem for $D_l$ as follows: For every $w \in V(D_l), v \in V(D)$ with $w \in V(P(v))$, let $y_w':=\frac{y_v}{(l-1)\vec{\chi}_f(D)+1} \ge 0$. Obviously, \begin{displaymath}\sum_{w \in V(D_l)}{y_w'}=\frac{\sum_{y \in V(D)}{ly_v}}{(l-1)\vec{\chi}_f(D)+1}=\frac{l\vec{\chi}_f(D)}{(l-1)\vec{\chi}_f(D)+1}.\end{displaymath} Furthermore, for each $B \in \mathcal{A}(D_l)$, we have

\begin{displaymath}\sum_{w \in B}{y_w'} \leq \sum_{v \in f(B)}{l\frac{y_v}{(l-1)\vec{\chi}_f(D)+1}}+\sum_{v \in V(D) \backslash f(B)}{(l-1)\frac{y_v}{(l-1)\vec{\chi}_f(D)+1}}\end{displaymath}
\begin{displaymath}=\frac{1}{(l-1)\vec{\chi}_f(D)+1}\left((l-1)\sum_{v \in V(D)}{y_v}+\underbrace{\sum_{v \in f(B)}{y_v}}_{\leq 1}\right) \leq 1.\end{displaymath}
Thus, the $y_w'$ are admissible for the program, which proves the first inequality. \n
For the reverse inequality, we want to show \begin{displaymath}\vec{\chi}_f(D_l) \leq \frac{l\vec{\chi}_f(D)}{(l-1)\vec{\chi}_f(D)+1} \qquad \textrm{ or equivalently }\end{displaymath} \begin{displaymath} \vec{\chi}_f(D) \ge \frac{\vec{\chi}_f(D_l)}{l-(l-1)\vec{\chi}_f(D_l)}.\end{displaymath} Notice that always $\vec{\chi}_f(D_l) \leq \vec{\chi}^\ast(D_l) \leq \frac{l\vec{\chi}^\ast(D)}{(l-1)\vec{\chi}^\ast(D)+1} < \frac{l}{l-1}$ and thus $l-(l-1)\vec{\chi}_f(D_l) > 0$. \n
Assume now that $y_w', w \in V(D_l)$ is an optimal solution of the dual program (\ref{dual}) for $D_l$, which means $\sum_{w \in V(D_l)}{y_w'}=\vec{\chi}_f(D_l).$ We define an instance of the dual program for $D$ according to \begin{displaymath}y_v:=\frac{\sum_{w \in V(P(v))}{y_w'}}{l-(l-1)\vec{\chi}_f(D_l)} \ge 0\end{displaymath} for each $v \in V(D)$. First of all, with this definition, we have \begin{displaymath}\sum_{v \in V(D)}{y_v}=\frac{\sum_{w \in V(D_l)}{y_w'}}{l-(l-1)\vec{\chi}_f(D_l)}=\frac{\vec{\chi}_f(D_l)}{l-(l-1)\vec{\chi}_f(D_l)}.\end{displaymath} For the above inequality, it thus suffices to verify that the $y_v$ define a legal instance for the dual program: Let $A \in \mathcal{A}(D)$ be arbitrary. Then
\begin{displaymath}\sum_{v \in A}{y_v}=\frac{\sum_{v \in A}{\sum_{w \in V(P(v))}{y_w'}}}{l-(l-1)\vec{\chi}_f(D_l)}=\frac{\sum_{w \in g(A)}{y_w'}}{l-(l-1)\vec{\chi}_f(D_l)}.\end{displaymath}
For each $v \in V(D) \backslash A$, we choose exactly one vertex $w_v \in V(P(v))$ with minimal value within $P(v)$ and consider the acyclic vertex subset $X:=V(D_l) \backslash \bigcup_{v \in V(D)\backslash A}{\{w_v\}}$ which contains $g(A)$. According to our choice of the $w_v$, we know that
\begin{displaymath}1 \ge \sum_{w \in X}{y_w'}=\sum_{w \in g(A)}{y_w'}+\sum_{v \in V(D) \backslash A}{\left(\sum_{w \in V(P(v)), w \neq w_v}{y_w'}\right)}\end{displaymath} \begin{displaymath}\ge \sum_{w \in g(A)}{y_w'}+\sum_{v \in V(D) \backslash A}{\frac{l-1}{l}\sum_{w \in V(P(v))}{y_w'}}=\left(\frac{l-1}{l}+\frac{1}{l}\right)\sum_{w \in g(A)}{y_w'}+\frac{l-1}{l}\sum_{w \in V(D_l) \backslash g(A)}{y_w'}\end{displaymath} \begin{displaymath}=\frac{l-1}{l}\vec{\chi}_f(D_l)+\frac{1}{l}\sum_{w \in g(A)}{y_w'}.\end{displaymath}
Multiplying the inequality with $l$ and subtracting $(l-1)\vec{\chi}_f(D_l)$ now yields that indeed \begin{displaymath}\sum_{v \in A}{y_v}=\frac{\sum_{w \in g(A)}{y_w'}}{l-(l-1)\vec{\chi}_f(D_l)} \le 1\end{displaymath} and thus $\vec{\chi}_f(D) \ge \frac{\vec{\chi}_f(D_l)}{l-(l-1)\vec{\chi}_f(D_l)}$ as claimed. \n
Finally, this proves $\vec{\chi}_f(D_l)=\frac{l\vec{\chi}_f(D)}{(l-1)\vec{\chi}_f(D)+1}$.
\end{itemize}
\noindent
If now $\vec{\chi}^\ast(D)=\vec{\chi}_f(D)$, we may conclude:
\begin{displaymath}\frac{l\vec{\chi}^\ast(D)}{(l-1)\vec{\chi}^\ast(D)+1}=\frac{l\vec{\chi}_f(D)}{(l-1)\vec{\chi}_f(D)+1}=\vec{\chi}_f(D_l) \leq \vec{\chi}^\ast(D_l) \leq \frac{l\vec{\chi}^\ast(D)}{(l-1)\vec{\chi}^\ast(D)+1}.\end{displaymath}
Thus all inequalities hold with equality, and this yields the last statement claimed in the theorem.
\end{proof}

The following is now an immediate consequence of the above:
\begin{theorem} \label{fraccomp}
Problem \ref{probfrac} is NP-complete for every real number $p >1, p \neq 2$.
\end{theorem}
\begin{proof}
The case $p>2$ was proved in Observation \ref{Gr2Fall}, so let now $p \in (1,2)$ be arbitrary. Then there is an $l \in \mathbb{N}, l \ge 2$ only dependent on $p$ such that $p \in (\frac{2l}{2l-1},\frac{l}{l-1})$. Choose such an $l$ and define $p':=\frac{p}{l-(l-1)p}$. Then $p'>2$ and thus, Problem \ref{probfrac} is NP-hard for~$p'$. However, since the function $x \rightarrow \frac{lx}{(l-1)x+1}$ is strongly increasing for positive values of $x$, we have for any digraph $D$ that $\vec{\chi}_f(D) \leq p'$ if and only if $\vec{\chi}_f(D_l)=\frac{l\vec{\chi}_f(D)}{(l-1)\vec{\chi}_f(D)+1} \leq \frac{lp'}{(l-1)p'+1}=p$. This thus provides a polynomial reduction of Problem \ref{probfrac} with $p'$ to the one with $p$, proving the NP-hardness (and thus -completeness) of the latter.
\end{proof}
\section{The Complexity of the Circular Vertex Arboricity} \label{ComplVertA}
In this section, we treat the analogue of the decision Problems \ref{probstar} and \ref{probfrac} for the circular vertex arboricity $\textrm{va}_c(G)$ of graphs.
\begin{problem} \label{probcircvert}
Let $p \ge 1$ be a fixed rational number. \n
Instance: A (multi-)graph $G$. \n Decide whether $\textrm{va}_c(G) \leq p$.
\end{problem}
It is easy to see that in the case that $p=\frac{k}{d} \ge 1$ is a rational number, any $(k,d)$-tree-colouring of a graph can be used as polynomially verifiable certificate for $\textrm{va}_c(G) \leq p$, and so the above decision problem for $p$ is contained in NP. \n
In order to prove complexity results, as in the case of the star dichromatic number, we could introduce a notion of circular homomorphisms between graphs analogous to Definition \ref{circhom} and consider corresponding homomorphism-colouring problems. However, unlike in the case of digraphs, no easy interpretation of the $(k,d)$-tree-colouring problem of a graph as such a homomorphism problem seems possible in general, which goes along with the fact that no simple canonical constructions of graphs with circular vertex arboricity $\frac{k}{d}$ for any pair $(k,d)$ similar to the circulant (di)graphs $\vec{C}(k,d),C(k,d)$ are known so far. \n
It is again easily observed that $va_c(G) = 1$ for any graph $G$ if and only if it is a forest, so the above decision problem is polynomially solvable for~$p=1$. \n
In the following we prove that similar to the cases of the circular and star dichromatic numbers, Problem \ref{probcircvert} is NP-complete for all rational numbers~$p>1$. We prepare the proof with the following observation.
\begin{lemma}
Let $(k,d) \in \mathbb{N}^2, k > d$. Let $I(k,d)$ denote the minimal size of a subset of $\mathbb{Z}_k$ which is not contained in a cyclic subinterval of size $d$. Then \begin{displaymath}I(k,d):=\min\left\{|A| \big\vert A \subseteq \mathbb{Z}_k, \forall i \in \mathbb{Z}_k : A \not\subseteq \{i,i+1,...,i+d-1\}\right\}=\left\lceil \frac{k}{k-d} \right\rceil. \end{displaymath}
\end{lemma}
\begin{proof}
The complements of the cyclic subintervals of $\mathbb{Z}_k$ of size $d$ are the cyclic subintervals of size $k-d$. Thus, if $A \subseteq \mathbb{Z}_k$ is not contained in a cyclic subinterval of length $d$, any two consecutive points in $A$ according to the cyclic ordering of $\mathbb{Z}_k$ must have cyclic distance at most $k-d$. Consequently, \mbox{$(k-d)|A| \ge k$ implying $|A| \geq \left\lceil \frac{k}{k-d} \right\rceil$}, and thus $I(k,d) \geq \left\lceil \frac{k}{k-d} \right\rceil$. On the other hand, we may define $\left\lceil \frac{k}{k-d} \right\rceil$ points in $\mathbb{Z}_k$ according to $a_i:=((k-d)i)$, for each $i \in \{0,\ldots, \left\lceil \frac{k}{k-d} \right\rceil-1\} \subseteq \mathbb{Z}_k$, and it is easily seen that $\left\{a_0,\ldots,a_{\left\lceil \frac{k}{k-d} \right\rceil-1}\right\}$ defines a set as required, proving $I(k,d) \leq \left\lceil \frac{k}{k-d} \right\rceil$.
\end{proof}

For any pair $(k,d) \in \mathbb{N}^2, k > d $ we now define a simple auxiliary graph $H(k,d)$ which has vertex set $\mathbb{Z}_k$ and in which a pair $i \neq j \in \mathbb{Z}_k$ of vertices is adjacent if and only if there is a subset $A \subseteq \mathbb{Z}_k$ not contained in any cyclic subinterval of size $d$ such that $\{i,j\} \subseteq A$ and $|A|=I(k,d)$. It is easy to see that whenever $\frac{k}{d} \ge 2$, $H(k,d)$ is just the circulant graph $C(k,d)$ defined in Definition \ref{circulant}. More generally, it follows from the definition that adjacency in $H(k,d)$ only depends on the circular distance of the respective vertices. Hence, $H(k,d)$ is always a circulant graph. For instance, $H(5,3)$ is the complete graph $K_5$, $H(6,4)$ is the disjoint union of two triangles and $H(8,5)$ admits an edge between vertices $i,j \in \mathbb{Z}_{8}$ if and only if $|i-j|_{8} \in \{2,3\}$.  \n
We are now prepared to prove the following NP-hardness result.
\begin{theorem}\label{mainvert}
For any rational number $p>1$, Problem \ref{probcircvert} is NP-complete.
\end{theorem}
\begin{proof}
The NP-membership of the problem was verified above. \n
So let now $1<p=\frac{k}{d}$ be arbitrary but fixed. We distinguish between the cases $p=2$ and $p \neq 2$. \n
Assume first that $p \neq 2$. We prove the claimed NP-hardness by describing a polynomial reduction of the $H(k,d)$-colouring problem (in terms of graph homomorphisms) to Problem \ref{probcircvert} with $p=\frac{k}{d}$. To do so, given any graph $G$ as an instance of the $H(k,d)$-colouring problem, we construct (in polynomial time) a graph $G_{k,d}$ of size polynomial in $|V(G)|$ and prove that $G$ maps to $H(k,d)$ if and only if $\textrm{va}_c(G_{k,d}) \leq \frac{k}{d}$, which is equivalent to $G_{k,d}$ admitting a $(k,d)$-tree-colouring. \n
The graph $G_{k,d}$ is obtained from $G$ by replacing any edge $e \in E(G)$ by a bunch of $2^k$ parallel paths of length $I(k,d)-1$ each connecting the end vertices of $e$. The vertex sets of different replacement-paths are disjoint except for common end vertices. \n
To prove the first direction of the claimed equivalence, assume there is a graph homomorphism $\phi:V(G) \rightarrow \mathbb{Z}_k=V(H(k,d))$. This means that for any edge $e=xy \in E(G)$, $\phi(x) \neq \phi(y)$ are contained in a subset $A(e) \subseteq \mathbb{Z}_k$ of size $I(k,d)$ which is not contained in a cyclic subinterval of $\mathbb{Z}_k$ of size $d$. We now define a colouring $c:V(G_{k,d}) \rightarrow \mathbb{Z}_k$ of $G_{k,d}$ as follows: Any vertex $v \in V(G_{k,d})$ originally contained in $G$ gets colour $c(v):=\phi(v)$. For any replacement-path $P$ of an edge $e \in E(G)$, we assign all the $I(k,d)-2$ elements of $A(e) \setminus \{\phi(x),\phi(y)\}$ to the $I(k,d)-2$ internal vertices of $P$ (in arbitrary order). We claim that this defines a $(k,d)$-tree-colouring of $G_{k,d}$: For any cycle $C$ in $G_{k,d}$, $V(C)$ contains the vertex set of a whole replacement-path of an edge $e \in E(G)$, and thus $c(V(C)) \supseteq A(e)$. As $A(e)$ is not contained in any cyclic subinterval of $\mathbb{Z}_k$ of size $d$, the same is true for $c(V(C))$. This proves the validity of $c$ as a $(k,d)$-tree-colouring and we conclude $\textrm{va}_c(G_{k,d}) \leq \frac{k}{d}$. \n
To prove the reverse implication, assume $\textrm{va}_c(G_{k,d}) \leq p=\frac{k}{d}$, i.e., $G_{k,d}$ admits a $(k,d)$-tree-colouring $c$. We define $\phi:V(G) \rightarrow \mathbb{Z}_k=V(H(k,d))$ by restriction of $c$ to the vertices originally contained in $G$. We claim that this defines a graph homomorphism. To prove this, let $e=xy \in E(G)$ be an edge. Any of the $2^k$ replacement-paths of $e$ in $G_{k,d}$ receives a non-empty subset of $\mathbb{Z}_k$ of colours according to $c$. Applying the pigeon-hole principle we find a pair $P_1(e) \neq P_2(e)$ of replacement-paths of $e$ such that $c(V(P_1(e)))=c(V(P_2(e)))$. As the union of $P_1(e)$ and $P_2(e)$ forms a cycle in $G_{k,d}$, according to the definition of a $(k,d)$-tree-colouring, it follows that $c(V(P_1(e))) \cup c(V(P_2(e)))=c(V(P_1(e)))$ is not contained in a cyclic subinterval of $\mathbb{Z}_k$ of size $d$ and is of size at most $|c(V(P_1(e)))|\leq|V(P_1(e))|=I(k,d)$. According to the definition of $I(k,d)$, this implies $|c(V(P_1(e)))|=I(k,d)$. Consequently, all the $I(k,d)$ colours assigned to the vertices of $P_1(e)$ are pairwise distinct, and thus, $\phi(x),\phi(y) \in c(V(P_1(e)))$ are distinct. According to the definition of $H(k,d)$, this implies that $\phi(x)\phi(y)$ forms an edge in $H(k,d)$, i.e., $\phi$ indeed is a graph homomorphism mapping $G$ to $H(k,d)$. \n
We finally conclude the correctness of the reduction. For the NP-hardness, it thus remains to verify that the $H(k,d)$-colouring problem is NP-hard. According to Theorem \ref{nesetril}, its suffices to prove that $H(k,d)$ is non-bipartite. As $p=\frac{k}{d} \neq 2$, we either have $\frac{k}{d}<2$, which implies $I(k,d)=\left\lceil \frac{k}{k-d} \right\rceil \ge 3$, and thus, $H(k,d)$ contains a clique of size at least $3$ and is thus not bipartite. Otherwise, we have $\frac{k}{d}>2$ and thus, $\chi(H(k,d))=\chi(C(k,d))=\lceil p \rceil \ge 3$. This finally yields the claimed NP-hardness (and thus -completeness) in the case $p \neq 2$. \N
In the remaining case of $p=2$, deciding Problem \ref{probcircvert} is the same as deciding whether a given graph $G$ fulfills $\textrm{va}(G) \leq 2$. However, it is not hard to see that a planar cubic 3-connected graph $G$ admits a Hamiltonian cycle if and only if its planar dual graph $G^\ast$ admits vertex arboricity at most $2$ (cf. \cite{hakimischmeichel}, \cite{Hoch}). Consequently, the NP-hardness of the decision problem \ref{probcircvert} in this case follows from the NP-hardness of the Hamiltonicity problem restricted to planar cubic 3-connected graphs. 
\end{proof}
\section{Conclusive Remarks}
The complexity results achieved in this paper together with the results by \cite{bojannp} clarify our view on fractional colouring parameters related to acyclic vertex sets in digraphs and graphs in terms of computational complexity. In contrast to the initial guess of the authors, deciding $p$-colourability for these notions remains NP-complete even for values of $p$ close to $1$. Looking at related notions such as the \emph{fractional and circular arboricity} of graphs, which can be computed in polynomial time using the Matroid Partitioning Algorithm (cf. \cite{kneser}, Chapter 5 and \cite{vanden}), those results show that circular notions of vertex- and edge-arboricity behave very differently with respect to complexity.

Theorem~\ref{fraccomp} furthermore does not treat the case $p=2$. Thus, we cannot rule out the possibility that there is some clever way to algorithmically decide whether a given digraph has fractional dichromatic number at most $2$. Still, the authors strongly believe that using other techniques, it should be possible to prove NP-completeness also in this case.

A natural question left open in this paper concerns restrictions of the treated decision problems to specialised inputs. An interesting special case consists of (simple) planar (di)graphs. It is clear that deciding the problems \ref{probstar}, \ref{probfrac} and \ref{probcircvert} will now be trivially polynomial-time solvable for large values of $p$, as for instance, simple planar digraphs are conjectured to be 2-dichromatic (\cite{NeuLara}), while an upper bound of $2.5$ for each of the three notions studied in this paper is known when restricting to simple planar (di)graphs (\cite{hochsteiner}, \cite{fracvert}).

It appears to be hard to use the reductions provided in this paper to achieve hardness results for planar inputs. This is mostly due to the fact that the complexity of $H$-colouring planar graphs is very poorly understood. While the $K_4$-colouring problem is trivially in $P$ (output true), only for few graphs $H$ such as odd cycles (cf. \cite{hcolouringplanar}) hardness results are known, while for many non-trivial graphs such as the \emph{Clebsch graph}, the $H$-colouring problem becomes solvable in polynomial time. Moreover, the $l$-split-operation used in Section \ref{secfrac} does not preserve planarity.

Still, we may deduce the following special cases:
\begin{theorem}
\noindent
\begin{itemize}
\item Deciding whether a given simple planar digraph $D$ fulfills $\vec{\chi}^\ast(D) \leq \frac{3}{2}$ is NP-complete.
\item Deciding whether a a given simple planar graph $G$ fulfills $\textrm{va}_c(G) \leq \frac{3}{2}$ is NP-complete.
\end{itemize}
\end{theorem}
\begin{proof}
Both problems are clearly in NP. Notice that $H_{\vec{C}(3,2)}=K_3$ for the graph defined in the proof of Theorem \ref{main} and $H(3,2)=K_3$ for the auxiliary graph defined in section \ref{ComplVertA}. It is easy to see that the digraph $D_G$ defined in the proof of Theorem \ref{main} as well as the graph $G_{3,2}$ as defined in the proof of Theorem \ref{mainvert} are both planar and simple for any planar and simple graph $G$. Deciding $K_3$-colourability of planar graphs can thus be polynomially reduced to each of the above decision problems. The NP-hardness of $3$-colourability of planar graphs (cf. \cite{3colouringplanar}) now yields the claim.
\end{proof}

It would be furthermore interesting to study the notion of circular homomorphisms in more detail. Natural questions consider for instance descriptions of the \emph{cores} of such homomorphisms, which enable the study of a corresponding \emph{homomorphism order}.
For graphs, this is a wide and active field of research, we refer to the book by \cite{hellnesbook} for a comprehensive survey of the topic. 

In this context, a graph $G$ is called a \emph{core} if it does not admit a graph homomorphism to a proper subgraph. Equivalently, one may define a core to be a graph $G$ such that every homomorphism $\phi:V(G) \rightarrow V(G)$ is a bijection. The interest in cores comes from their role as minimal representatives of homomorphic equivalence classes of graphs. \cite{uncol} considered the following corresponding definition of digraph cores: A digraph $D$ is called a \emph{core}, if every acyclic homomorphism $\phi:V(D) \rightarrow V(D)$ of $D$ to itself is a bijection. Similarly, we define a digraph $D$ to be a \emph{circular core} if any circular homomorphism $\phi:V(D) \rightarrow V(D)$ is bijective. We want to conclude the discussion of circular homomorphisms with some first observations concerning this notion. 
\begin{proposition}
\noindent
\begin{itemize}
\item A graph $G$ is a core if and only if $S(G)$ is a circular core.
\item If $D$ is a circular core, then $D$ is a core with respect to acyclic homomorphisms.
\item For any integers $k \ge d \ge 1$, the circulant digraph $\vec{C}(k,d)$ is a core if and only if $k$ and $d$ are coprime.
\end{itemize}
\end{proposition}
\begin{proof}
\noindent
\begin{itemize}
\item This is a direct consequence of Proposition \ref{symmetricor}.
\item This follows from Proposition \ref{acyclich}.
\item Assume for the first direction that $\textrm{gcd}(k,d)=l>1$, let $k':=\frac{k}{l}, d':=\frac{d}{l}$ and consider the mapping $\phi:\mathbb{Z}_k \rightarrow \mathbb{Z}_{k'}$ defined by $\phi(i):=\lfloor \frac{i}{l} \rfloor$ for all $i \in \{0,\ldots,k-1\} \simeq \mathbb{Z}_k$. We claim that this defines a circular homomorphism from $\vec{C}(k,d)$ to $\vec{C}(k',d')$. However, $\phi$ is easily seen to be an acylic $(k',d')$-colouring of $\vec{C}(k,d)$, and according to Proposition \ref{ckd}, this already means that $\phi$ is a circular homomorphism. As $\vec{C}(k',d')$ is isomorphic to the proper subdigraph of $\vec{C}(k,d)$ induced by the vertices $il, i=0,\ldots,k'-1$, this proves that $\vec{C}(k,d)$ is circularly homomorphic to a proper induced subdigraph and thus no circular core.

To prove the reverse, let $\textrm{gcd}(k,d)=1$ and assume that contrary to the assertion, $\vec{C}(k,d)$ admits a circular homomorphism $\phi:\vec{C}(k,d) \rightarrow \vec{C}(k,d)$ which is not bijective. Let $D$ be the subdigraph of $\vec{C}(k,d)$ induced by $\textrm{Im}(\phi)$. Then according to Proposition \ref{homandcol}, we have $\frac{k}{d}=\vec{\chi}^\ast(\vec{C}(k,d)) \leq \vec{\chi}^\ast(D) \leq \vec{\chi}^\ast(\vec{C}(k,d))=\frac{k}{d}$. However, referring to Theorem \ref{StarDichBasic}, (i), we also know that $\frac{k}{d}=\vec{\chi}^\ast(D)$ can be represented as a fraction with numerator at most $|V(D)|<|V(\vec{C}(k,d))|=k$. This finally contradicts the assumption that $k$ and $d$ are coprime, and we deduce the claimed equivalence.
\end{itemize}
\end{proof}

Moreover, it turns out that many famous theorems and problems for graph homomorphisms have directed analogues in terms of circular homomorphisms. For instance, the following is a well-known conjecture on homomorphism bounds for planar graphs of large girth:
\begin{conjecture}[cf. \cite{circflows}]
Any planar graph $G$ of girth at least $4k$ admits a homomorphism to $C_{2k+1}$. In other words, $\chi_c(G) \leq \frac{2k+1}{k}$.
\end{conjecture}
This conjecture contains as a special case for $k=1$ the famous Theorem of Gr\"{o}tzsch stating that every planar $C_3$-free graph admits a homomorphism to $C_3$, i.e., is $3$-colourable. The following directed analogue of this conjecture was posed by \cite{hochsteiner}:
\begin{conjecture}
Any planar digraph $D$ of directed girth at least $k \ge 3$ admits a circular homomorphism to $\vec{C}_{k-1}$. In other words, $\vec{\chi}^\ast(D) \leq \frac{k-1}{k-2}$.
\end{conjecture}
The case $k=3$ here corresponds to the so-called \emph{2-Colour-Conjecture} posed by Neumann-Lara, while $k=4$ means the following directed analogue of Gr\"{o}tzsch' result: Every oriented planar $\vec{C}_3$-free digraph admits a circular homomorphism to $\vec{C}_3$. An analogue of the above conjecture using acyclic homomorphisms however is not possible: There are planar digraphs with arbitrarily large digirth but circular dichromatic number $2$. Examples are again given by oriented wheels of arbitrary size with a dominating source whose deletion results in a directed cycle.

Closely related to the fractional chromatic number of a graph is the notion of \emph{$b$-colourings} of graphs defined as follows: For fixed $b,k \in \mathbb{N}$ with $b \leq k$, a $(k,b)$-colouring of a graph $G$ is an assignment of subsets $B_v \subseteq \{1,\ldots,k\}, v \in V(G)$ of size $b$ to the vertices in such a way that $B_x \cap B_y=\emptyset$ for adjacent vertices $x,y \in V(G)$. The \emph{$b$-chromatic number} $\chi_b(G)$ of $G$ is then defined as the least $k$ for which a $(k,b)$-colouring exists. It is a well-known fact that $\chi_f(G)=\lim_{b \rightarrow \infty}{\frac{\chi_b(G)}{b}}=\inf_{b \ge 1}{\frac{\chi_b(G)}{b}}$ for any loopless graph $G$ (cf. \cite{kneser}).

Regarding the relation of $b$-colourings of graphs and their fractional chromatic number, it is natural to ask for a similar notion corresponding to the fractional dichromatic number. Such a notion was defined by \cite{doct} as follows: Given any natural numbers $k \ge b$, a \emph{$b$-tuple $k$-colouring} of a digraph $D$ is defined to be an assignment of subsets $B_v \subseteq \{1,\ldots,k\}, v \in V(D)$ of size $b$ to the vertices in such a way that for any $i \in \{1,\ldots,k\}$, the subdigraph of $D$ induced by the vertices having $i$ in their colour set, is acyclic. It can again be shown (cf. \cite{doct}, Theorem 5.1) that for any digraph $D$, $\vec{\chi}_f(D)$ is the minimal fraction $\frac{k}{b}$ for which a $b$-tuple $k$-colouring of $D$ exists. \\
For $b$-colourings of graphs, a very useful reformulation in terms of homomorphisms to the Kneser graphs $K(k,b)$ is known: For any $k, b \in \mathbb{N}, k \ge b$, a graph $G$ admits a $(k,b)$-colouring if and only if there is a graph homomorphism mapping $G$ to $K(k,b)$ (cf. \cite{kneser}, Proposition 3.2.1).

It would be interesting to find a directed version of Kneser graphs yielding an analogous reformulation of $b$-tuple colourings of digraphs. To be more precise, we conclude with the following question.
\begin{problem}
For given $k, b \in \mathbb{N}$ with $k \ge b$, is there a directed graph $\vec{K}(k,b)$ with vertex set $\binom{[k]}{b}$ such that the following holds?\\
The subdigraph of $\vec{K}(k,b)$ induced by any $\{B_1,\ldots,B_l\} \subseteq \binom{[k]}{b}$ is acylic if and only if $\bigcap_{i=1}^{l}{B_i} \neq \emptyset$.
\end{problem}
For instance, whenever $k=b+1$, such a digraph is given by the directed cycle of length $b+1$ whose vertices are associated with the sets $\{1,\ldots,b+1\}\setminus\{i\}$ for $i=1,\ldots,b+1$.

If such a ``directed Kneser graph'' $\vec{K}(k,b)$ for the parameters $(k,b)$ exists, we have the following reformulation of $b$-tuple $k$-colourings in terms of circular homomorphisms: \\
A digraph $D$ is $b$-tuple $k$-colourable if and only if there is a circular homomorphism mapping $D$ to $\vec{K}(k,b)$.

\bibliographystyle{abbrvnat}
\bibliographystyle{alpha}
\bibliography{references}
\label{sec:biblio}

\end{document}